\documentclass{article}
\usepackage{amsfonts, amsmath, amscd, amsthm, bm}

\usepackage{mathpazo}
\usepackage{euler}
\usepackage{a4wide}

\usepackage[all]{xy}
\newtheorem{defn}{Definition}
\newtheorem{thm}{Theorem}
\newtheorem{prop}{Proposition}
\newtheorem{lem}{Lemma}
\newtheorem{cor}{Corollary}
\newtheorem*{cor*}{Corollary}
\newtheorem*{corA}{Corollary \ref{cor-RST-planar-case}}
\theoremstyle{remark}
\newtheorem{rem}{Remark}

\newcommand{\QQ}{\mathbb{Q}}
\newcommand{\PP}{\mathbb{P}}

\newcommand{\OO}{\mathcal{O}}
\newcommand{\catname}[1]{{\bm{#1}}}
\newcommand{\fname}[1]{{\normalfont\textbf{#1}}}
\newcommand{\Spec}[1]{{\operatorname{Spec}(#1)}}
\newcommand{\Hom}{{\operatorname{Hom}}}
\newcommand{\Ext}{{\operatorname{Ext}}}
\newcommand{\Hilb}{{\operatorname{Hilb}}}
\newcommand{\Quot}{{\operatorname{Quot}}}
\newcommand{\Pic}{{\operatorname{Pic}}}
\newcommand{\mult}{{\normalfont\textrm{mult}}}
\newcommand{\ord}{{\normalfont\textrm{ord}}}
\newcommand{\Q}{{\operatorname{Q}}}
\newcommand{\PS}{{[\![ t ]\!]}}
\title{A Cancellation Theorem for Segre Classes}
\author{Daniel Lowengrub}

\begin{document}
\maketitle

\section{Introduction}
In \cite[4]{Fulton}, Fulton defines the notion of the Segre class $s(X,Y)\in A_*X$ of a closed embedding of schemes $X\rightarrow Y$ over a field $k$. As in Fulton, all of our schemes are finite type over a ground field $k$ which may be of arbitrary characteristic unless stated otherwise.
 
The Segre class allows us to measure the way in which $X$ sits inside $Y$, and is functorial for sufficiently nice maps (\cite[4.2]{Fulton}).
One important case is the embedding of a closed point, for which the Segre class gives us its multiplicity.

Suppose we have an embedding $X\rightarrow Y$ and the schemes in question embed into a simpler space $Z$ such that $Y\rightarrow Z$ is a regular embedding. 
For example, $Y$ could be an intersection of hyper-surfaces in $Z=\PP^n_k$. 
In this setup, it is natural to ask whether we can calculate $s(X,Y)$, assuming that we understand $s(X,Z)$ and $c(N_YZ)$ where $N_YZ$ is the normal bundle of the regular embedding $Y\rightarrow Z$. 
In other words, can we deduce intersection theoretic invariants of the possibly complicated embedding $X\rightarrow Y$, from the hopefully simpler embeddings into $Z$?

Fulton provides the answer to this question in one very special case:
\begin{prop}(\cite[4.2.7]{Fulton})\label{prop-easy-case}
Let $X$ be a closed subscheme of $Y$, $E$ a vector bundle on $Y$, with $Y$ embedded into $E$ as the zero section. Then,
\[
s(X,Y) = c(E|_X)\cap s(X,Z)
\]
\end{prop}

It it natural to wonder whether this is true in general, when we replace $E$ by the normal bundle $N_YZ$. 
Unfortunately, this is false in extremely simple cases. See \cite[4.2.8]{Fulton} for an example.

Despite this failure in the general case, it is intuitive that the statement should hold when the embedding of $Y$ into $Z$ ``looks like'' a zero section of a vector bundle. 
For instance, it seems plausible that we could replace the condition of $Y\rightarrow Z$ being a zero section, with the condition that $Y\rightarrow Z$ have some sort of tubular neighborhood.

The first generalization that we'll prove is the following.

\begin{prop}\label{prop-smooth-case}
Let $X$ be a finite type $k$-scheme let $Y$ and $Z$ be smooth $k$-schemes. 
Suppose that we have a closed embedding $X\rightarrow Y$ and regular embedding $Y\xrightarrow{f} Z$.
Then,
\[
s(X,Y) = c(N_YZ|_X)\cap s(X,Z)
\]
\end{prop}

Our primary result of this paper is the following strengthening of proposition \ref{prop-smooth-case} in the case of characteristic zero.

\begin{thm}\label{thm-main-theorem}
Let $k$ be a field of characteristic zero, and let $X$, $Y$ and $Z$ be varieties over $k$. 
Suppose that we have a closed embedding $X\rightarrow Y$ and regular embedding $Y\xrightarrow{f} Z$ that formally locally looks like a section $Y\xrightarrow{s}P$ of some smooth map $P\rightarrow Y$. 
By this we mean that for any closed point $y=\Spec{k(y)}\rightarrow Y$, there exists a smooth map $P\rightarrow Y$, a section $Y\xrightarrow{s}P$, and an isomorphism $\varphi$
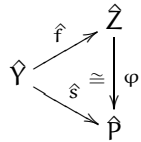
\[
\xymatrix @ R=0.5pc {
  &  \hat{Z} \ar[dd]^{\varphi}_{\cong} \\
\hat{Y} \ar[ru]^{\hat{f}} \ar[rd]^{\hat{s}}& \\
  & \hat{P}
}
\]

Where $\hat{Y}$, $\hat{Z}$ and $\hat{P}$ are the formal completions of $Y$, $Z$ and $P$ along the points $y$, $f(y)$ and $s(y)$ respectively.

Then,
\[
s(X,Y) = c(N_YZ|_X)\cap s(X,Z)
\]
\end{thm}

This property of $f$ seems to capture the notion of a tubular neighborhood, as we can formally locally view it as a section of a vector bundle. 

We call it a ``cancellation theorem'' since intuitively, it tells us that in order to calculate $s(X,Y)$, we can first calculate $s(X,Z)$, and then ``cancel out'' the contribution of the embedding $Y\rightarrow Z$. 

The theorem can be useful in practice, since there exist algorithms which have been implemented in Macaulay2 that can compute the Segre class of an embedding $X\rightarrow \PP^n_k$ \cite{EJP}. 
Our theorem allows us to use this algorithm in order to compute $s(X,Y)$ when both $X$ and $Y$ are subschemes of $\PP^n$ and the embedding $Y\rightarrow \PP^n$ satisfies our condition. 
In particular, proposition \ref{prop-smooth-case} allows us to compute $s(X,Y)$ when $Y$ is a smooth intersection of hypersurfaces in $\PP^n_k$ since we then know $c(N_Y\PP^n_k)$ as well. 

Theorem \ref{thm-main-theorem} is particularly useful when the spaces $Y$ and $Z$ represent functors and the embedding $Y\rightarrow Z$ corresponds to a natural transformation of these functors. In this case, the formal neighborhoods pro-represent local deformation functors which are typically easy to describe. 

As an example of this, in section \ref{sec:RST-example} we will use theorem \ref{thm-main-theorem} in order to deduce a generalization of the Riemann Kempf formula to integral curves. As a corollary, we obtain the following generalization of the Riemann singularity theorem which was conjectured in \cite{CMK}:

\begin{corA}
Let $k$ be an algebraically closed field of characteristic $0$ and let $X$ be a projective integral curve of arithmetic genus $p$ over $k$ with at most planar singularities.
Let $x$ be a $k$-point of the compactified Jacobian $P_0$ corresponding to an rank-1 torsion free sheaf $\mathcal{I}$.
Let $\Theta$ denote the image of the Abel Jacobi map $\mathcal{A}^{p-1}_{\omega}$ in $P_0$. Then
\[
\mult_x\Theta = \mult_xP_{0} \cdot (h^1(X,\mathcal{I}) - 1)
\]
where $h$ is the first Chern class of the canonical bundle on $(\mathcal{A}^{p-1}_{\omega})^{-1}(x)\cong\PP^r_k$ and $r = h^1(X,\mathcal{I}) - 1$.
\end{corA}

In that section, we use theorem \ref{thm-main-theorem} in order to study the change in various Segre classes as we increase the degree of our Quot schemes. This allows us to deduce facts about the Abel Jacobi map in small degrees from the simpler Abel Jacobi maps in high degrees.

One obstacle that arises in the proof of theorem \ref{thm-main-theorem} is that the stated condition is inherently local, and schemes with non-trivial Chow rings can frequently be covered by affine opens with trivial Chow rings, rendering vacuous local intersection theoretic statements.
Intuitively we would like to glue these local bundles together to a bundle $E\rightarrow Y$ such that $Y\xrightarrow{f}Z$ is the composition of the zero section $Y\rightarrow E$ and an open embedding $E\rightarrow Z$, but this is too much to hope for.

We circumvent this problem in two stages. The first stage is to prove the theorem in the case where $Y$ and $Z$ are smooth (in which case the local property comes for free).

We then use Hironaka's functorial resolution of singularities in order to reduce the theorem to the smooth case. 
The key part of this step is that in some sense, the unique assignment $\fname{BR}:\catname{Var}\rightarrow\catname{Set}$ allows us to package the local data that we get from the formal properties of $f$ into a morphism of \emph{smooth} schemes $Y'\rightarrow Z'$, whose normal bundle plays the role of our elusive tubular neighborhood bundle $E$. 
Furthermore, since these schemes will map properly and birationally onto $Y\rightarrow Z$, the functorial properties of Fulton's intersection theory allow us to reduce the theorem to the smooth case.

\subsection{Acknowledgments}
In this section we will present a hopefully accurate historic account of the development of the ideas in this paper. In \cite{CMK}, Sebastian Casalania-Martin and Jesse Kass found a generalization of the Riemann Singularity theorem to nodal curves and conjectured what the general formula should be. 
After looking at this, my advisor Vivek Shende noticed that one could arrive at the conjecture by formally following Fulton's proof of the classical theorem \cite[4.3.2]{Fulton}, but that the step relating lower order Hilbert schemes of points to the higher order ones was no longer true for general, or even regular, embeddings of singular schemes.
After trying a couple of different ways of circumventing this issue, it became apparent that the best sort of theorem that one could hope for was a cancellation theorem such as in theorem \ref{thm-main-theorem}.
A few months later, I went to a talk by Dan Edidin at the Stanford Algebraic Geometry Seminar where he discussed his results in \cite{Edidin}. 
One of the theorems in that paper seemed to share certain technical aspects with the cancellation of Segre classes problem, so after the talk I asked him how he resolved them.
During that conversation he told me about the trick of combining Hironaka's theorem with the Artin approximation theorem which is ultimately a key ingredient of this paper.

\section{The Smooth Case}
In this section we prove proposition \ref{prop-smooth-case}. It seems to be easier to prove the following slightly stronger statement.

\begin{prop}\label{prop-smooth-case-stronger}
Let $X\rightarrow Y$ be a closed embedding and $Y\rightarrow Z$ a regular embedding where $X$ and $Y$ are $k$-schemes and $Z$ is a smooth $k$-scheme. Then, the class
\[
c(N_YZ|_X)\cap s(X,Z)
\]
is independent of the regular embedding $Y\rightarrow Z$ and the smooth scheme $Z$.
\end{prop}

The class in question is reminiscent of Fulton's canonical class, and indeed, the proof of independence is similar. 

Note that in the case where $X=Y$ and $Y\rightarrow Y$ is the identity map, the proposition is trivial since then $s(X,Z)$ is the inverse of $c(N_YZ)$. On the other hand, when $X\rightarrow Y$ is also a regular embedding and $Y\rightarrow Z$ is the identity map, then $c_F(X)$ is equal to the product of $c(X)$ and the class in question, so the proposition becomes equivalent to the well definedness of the Fulton canonical class. The proposition in the general case is a kind of interpolation between these extremes.

It would be interesting to see if the methods in the second section could be extended to remove the requirement that $Z$ be smooth.

\begin{proof}
Let $X\xrightarrow{f}Y$ be a closed embedding of $k$-schemes, $Z_1$ and $Z_2$ smooth $k$-schemes and $Y\xrightarrow{g_i}Z_i$ regular embeddings. We want to show that
\[
c(N_YZ_1)\cap s(X,Z_1)=c(N_YZ_2)\cap s(X,Z_2)
\]

As in \cite[4.2.6]{Fulton}, since $Z_1$ and $Z_2$ are dominated by the smooth schemes $Z_1\times Z_2$, by replacing $Z_1$ with this product and replacing $g_1$ with the induced map, we can assume without loss of generality that there exists a smooth map $Z_1\xrightarrow{\rho}Z_2$ such that the following diagram commutes:
\[
\xymatrix @ R=0.5pc {
  &  &  Z_1 \ar[dd]^\rho \\
X \ar[r]^f & Y \ar[ru]^{g_1} \ar[rd]_{g_2} \\
  &  &  Z_2
}
\]

Now, by \cite[4.2.6]{Fulton}, we have the following short exact sequences of cones (in the sense of \cite[4.1.6]{Fulton}).

\begin{itemize}
\item $0 \rightarrow g_1^{*}T_{\rho} \rightarrow C_YZ_1 \rightarrow C_YZ_2 \rightarrow 0$
\item $0 \rightarrow (g_1\circ f)^{*}T_{\rho} \rightarrow C_XZ_1 \rightarrow C_XZ_2 \rightarrow 0$
\end{itemize}

Note that in this case, $C_YZ_i=N_YZ_i$. By pulling back the first sequence of bundles, we obtain:
\[
0 \rightarrow (g_1\circ f)^{*}T_{\rho} \rightarrow f^*N_YZ_1 \rightarrow f^*N_YZ_2 \rightarrow 0
\]

Therefore, by \cite[4.1.6]{Fulton} and the definition of the Segre class we have:
\begin{itemize}
\item $s(f^*N_YZ_2) = c((g_1\circ f)^{*}T_{\rho})\cap s(f^*N_YZ_1)$
\item $s(X,Z_2) = c((g_1\circ f)^{*}T_{\rho})\cap s(X,Z_1)$
\end{itemize}
Since by definition the Chern class is the inverse of the Segre class, the first equality implies that
\[
c((g_1\circ f)^{*}T_{\rho})  = c(f^*N_YZ_1)\cap s(f^*N_YZ_2)
\]

By using this equality to replace $c((g_1\circ f)^{*}T_{\rho})$ the second equation listed above, we get the result.
\end{proof}

We can now easily deduce proposition \ref{prop-smooth-case}. Recall that by proposition \ref{prop-easy-case}, the theorem holds in the case where $Y\rightarrow Z$ is the zero section of a vector bundle. The theorem for the smooth case now follows by applying proposition \ref{prop-smooth-case-stronger} to the case where $Y\rightarrow Z_1$ is the map $Y\rightarrow Z$ in the theorem, and $Y\rightarrow Z_2$ is the zero section of the bundle $N_YZ$.

\section{Reduction to the Smooth Case}
As usual in intersection theory, it would be enough to dominate the map $f$ by a map of smooth schemes. In other words, we want to find smooth schemes $M$ and $N$ which fit into the following fiber diagram
\[
\xymatrix{
M \ar[r] \ar[d] &  N \ar[d] \\
Y \ar[r]^f & Z
}
\]

such that the vertical maps are proper birational maps. We formalize this in the following lemma.

\begin{lem}\label{lem-reduction-is-sufficient}
  Let $X\rightarrow Y$ be a closed embedding and $Y\xrightarrow{f} Z$ a regular embedding where $X$, $Y$ and $Z$ are $k$-schemes.
  Suppose we had a regular embedding $\tilde{Y}\xrightarrow{\tilde{f}}\tilde{Z}$ of smooth $k$-schemes together with a fiber diagram
\[
\xymatrix{
\tilde{Y} \ar[r]^{\tilde{f}} \ar[d]^h &  \tilde{Z} \ar[d]^g \\
Y \ar[r]^f & Z
}
\]

such that $h$ and $g$ are proper and birational. Then
\[
s(X,Y) = c(N_YZ|_X)\cap s(X,Z)
\]

\end{lem}

\begin{proof}
Consider the extended fiber diagram

\[
\xymatrix{
\tilde{X} \ar[r] \ar[d]^\pi & \tilde{Y} \ar[r]^{\tilde{f}} \ar[d]^h &  \tilde{Z} \ar[d]^g \\
X \ar[r] & Y \ar[r]^f & Z
}
\]

By \cite[4.2]{Fulton}, we know that $\pi_*s(\tilde{X},\tilde{Y}) = s(X,Y)$ and that $\pi_*s(\tilde{X},\tilde{Z})=s(X,Z)$. 
Furthermore, by proposition \ref{prop-smooth-case}, we know that 
\[
s(\tilde{X},\tilde{Y}) = c(N_{\tilde{Y}}\tilde{Z}|_{\tilde{X}})\cap s(\tilde{X},\tilde{Z})
\]

So by proper base change, it suffices to show that in the current scenario, $h^*N_YZ\cong N_{\tilde{Y}}\tilde{Z}$.

To show this, consider the following conormal sequences together with the base change maps
\[
\xymatrix{
         &  h^*(N_YZ)^{\vee} \ar[r]              & h^*\Omega_Z \ar[r] \ar[d]         & h^*\Omega_Y \ar[r] \ar[d]      & 0 \\
0 \ar[r] &  (N_{\tilde{Y}}\tilde{Z})^{\vee} \ar[r] & \Omega_{\tilde{Z}} \ar[r] & \Omega_{\tilde{Y}} \ar[r]     & 0
}
\]

By commutativity, we get a map $h^*(N_YZ)^{\vee} \rightarrow (N_{\tilde{Y}}\tilde{Z})^{\vee}$ which is an isomorphism since the bottom row is a short exact sequence of vector bundles. Taking the dual of this map give us the desired isomorphism.
\end{proof}

We now use a combination of Artin approximation and Hironaka's functorial resolution of singularities to produce a map $\tilde{Y}\rightarrow\tilde{Z}$ which satisfies the requirements of the lemma.

For convenience, we state the relevant proposition from \cite{Artin}.

\begin{thm}(\cite[2.6]{Artin})\label{thm-artin-approx}
Let $X_1$ and $X_2$ be finite type $S$-schemes and let $x_i\in X_i$ be points. 
If the complete local rings $\hat{\OO}_{X_i,x_i}$ are $\OO_S$-isomorphic, then $X_1$ and $X_2$ are locally isomorphic for the etale topology.
By this we mean that there exists a diagram of etale maps
\[
\xymatrix{
   &  X'  \ar[rd] \ar[ld] & \\
X_1 &  &  X_2
}
\]

and a point $x'\in X'$ that maps to $x_1$ and $x_2$ respectively and induces an isomorphism of function fields. 
\end{thm}

We will need a slight generalization of this theorem. The letters denoting the schemes have been modified in order to clarify the application to our current question.
\begin{lem}\label{lem-common-etale-neighborhood}
Let $Y\rightarrow Z_1$ and $Y\rightarrow Z_2$ be maps of finite type $S$-schemes and let $y\in Y$ be a point that maps to $z_i\in Z_i$. Suppose there exists an isomorphism $\varphi$ of formal neighborhoods:

\[
\xymatrix @ R=0.5pc {
  &  \hat{Z_1} \ar[dd]^{\varphi}_{\cong}   \\
\hat{Y} \ar[ru]^{\hat{f}} \ar[rd]^{\hat{s}} & \\
  & \hat{Z_2}
}
\]

Where $\hat{Y}$ and $\hat{Z_i}$ are the formal completions of $Y$ and $Z_i$ and along the points $y$ and $z_i$ respectively.
Then the maps $Y\rightarrow Z_1$ and $Y\rightarrow Z_2$ have a common etale neighborhood. 

By this we mean that we have $k$-schemes $U$ and $V$ together with the following two diagrams:
\[
\xymatrix{
V \ar[r]^{\beta_1} \ar[dr]_{\varphi}    &   V_1 \ar[r]^{\alpha_1} \ar[d]^{\gamma_1}  &  U \ar[d]^{\delta_1} &  &  
V \ar[r]^{\beta_2}                    &   V_2 \ar[r]^{\alpha_2} \ar[d]^{\gamma_2}  &  U \ar[d]^{\delta_2}   \\
                                   &  Y \ar[r]^{f_i}   &   Z_1   &   &
                                   &  Y \ar[r]^{f_2}   &   Z_2
}
\]

such that $\beta_i$, $\gamma_i$, $\delta_i$ and $\phi=\gamma_1\beta_1$ are etale, $\alpha_1\beta_1 = \alpha_2\beta_2$, and $V$ is an etale neighborhood of $y$.

\end{lem}

\begin{proof}
By Artin approximation (theorem \ref{thm-artin-approx}), there exists a scheme $U$ satisfying the conditions in the theorem. 
For ease of notation, we define $V_i=Y\times_{Z_i}U$ and $v_i=(y,u)\in V_i$. 
Since the natural maps $V_1 \rightarrow Y$ are etale, by passing again to formal neighborhoods we see that $\hat{\OO}_{V_i,x_i}$ are $\OO_U$
isomorphic. The theorem then follows from a second application of Artin approximation.
\end{proof}

We will now combine this lemma with Hironaka's functorial resolution of singularities in order to reduce theorem \ref{thm-main-theorem}
to the smooth case. Or in other words, in order to achieve the conditions in lemma \ref{lem-reduction-is-sufficient}.

Before starting we'll recall the precise statement of functorial resolutions from \cite{Kollar}. Note that we are only stating the subset of
the theorem that we'll be using.

\begin{thm}(\cite[3.36]{Kollar})\label{thm-resolution-singularities}
Let $k$ be a field with characteristic zero. Then there exists an assignment $\fname{BR}$ from finite type $k$-schemes to schemes such that:
\begin{itemize}
\item $\fname{BR}(X)\rightarrow X$ may be constructed from $X$ by a finite sequence of blowups.
\item $\fname{BR}(X)$ is smooth.
\item $\fname{BR}$ commutes with pullbacks along smooth morphisms.
\end{itemize}
\end{thm}

Note that the map $\fname{BR}(X)\rightarrow X$ is part of the data of $\fname{BR}(X)$ as it is simply the composition of blowups.
See \cite{Kollar} for details on the construction.

We now apply this theorem to our situation.

\begin{prop}\label{lem-reduction-to-smooth}
Let $k$ be a field of characteristic zero, and let $X$, $Y$ and $Z$ be varieties over $k$. 
Suppose that we have a closed embedding $X\rightarrow Y$ and regular embedding $Y\xrightarrow{f} Z$ that formally locally looks like a section $Y\xrightarrow{s}P$ of some smooth map $P\xrightarrow{\pi} Y$ as in theorem \ref{thm-main-theorem}.
Then there is a regular embedding $\tilde{Y}\xrightarrow{\tilde{f}}\tilde{Z}$ of smooth schemes over $k$ together with a fiber diagram
\[
\xymatrix{
\tilde{Y} \ar[r]^{\tilde{f}} \ar[d]^h &  \tilde{Z} \ar[d]^g \\
Y \ar[r]^f & Z
}
\]

such that $h$ and $g$ are proper and birational.
\end{prop}

\begin{proof}
By lemma \ref{lem-common-etale-neighborhood}, for each point $y\in Y$ 
we have $k$-schemes $U$ and $V$ together with the following two diagrams:
\[
\xymatrix{
V \ar[r]^{\beta_1} \ar[dr]_{\varphi}    &   V_1 \ar[r]^{\alpha_1} \ar[d]^{\gamma_1}  &  U \ar[d]^{\delta_1} &  &  
V \ar[r]^{\beta_2}                    &   V_2 \ar[r]^{\alpha_2} \ar[d]^{\gamma_2}  &  U \ar[d]^{\delta_2}   \\
                                   &  Y \ar[r]^f   &   Z   &   &
                                   &  Y \ar[r]^s   &   P \ar@/^/[l]^{\pi}
}
\]

such that all vertical and diagonal arrows are etale, $\alpha_1\beta_1 = \alpha_2\beta_2$, and $V$ is an etale neighborhood of $y$.

Now, since $\pi$ is smooth, by functorial resolution of singularities (theorem \ref{thm-resolution-singularities}) we get that
\[
s^*\fname{BR}(P) = s^*\pi^*\fname{BR}(Y) = \fname{BR}(Y)
\]

By combining this with another application of the resolution of singularities theorem for the vertical etale maps, 
we obtain the following chain of equalities:

\[
\xymatrix @ R=0.3pc {
(f\varphi)^*\fname{BR}(Z) &  =  &  (\delta_1\alpha_1\beta_1)^*\fname{BR}(Z) \\
                &  =  &   (\alpha_1\beta_1)^*\fname{BR}(U) \\
                &  =  &   (\alpha_2\beta_2)^*\fname{BR}(U) \\
                & =  &    (\delta_2\alpha_2\beta_2)^*\fname{BR}(P) \\
                & =  &    (s\gamma_2\beta_2)^*\fname{BR}(P) \\
                & =  &  \fname{BR}(V)
}
\]

Therefore, the schemes $\tilde{Z}=\fname{BR}(Z)$ and $\tilde{Y}=Y\times_Z\tilde{Z}$ satisfy the requirements of the proposition.
To see this, note that by theorem \ref{thm-resolution-singularities}, $\fname{Z}$ is smooth and $\fname{BR}\rightarrow Z$ is proper birational map.
Furthermore, by the above chain of inequalities, we see that $\tilde{Y}$ is etale locally smooth, and that the map $\tilde{Y}\rightarrow Y$ is birational etale locally on the base. It is also proper since proper maps are preserved by base chance.
\end{proof}

Theorem \ref{thm-main-theorem} now follows immediately from this proposition combined with lemma \ref{lem-reduction-is-sufficient}.

\section{Example: Riemann Singularity Theorem}\label{sec:RST-example}

\subsection{Preliminaries}
\subsubsection{The Compactified Picard Scheme}\label{sec-AK-reference}
We'll start by recalling the definition of the compactified Picard scheme in \cite{AK}, together with some axillary notation.
Readers familiar with this paper are encouraged to skip to this section.

The discussion in the \cite{AK} takes place in the general setting of a proper morphism $X\xrightarrow{f}S$, whereas we are primarily concerned with integral curves over a field.
Nevertheless, we'll start by describing the general case since it isn't any more difficult than the curve case, and it will allow us to provide cleaner arguments later on.

Let $X\xrightarrow{f}S$ be a finitely presented morphism of schemes and let $\mathcal{F}$ be a locally finitely presented $\OO_X$-module. 
Altman and Kleiman define a collection of functors associated to such data, which are related to one another by the Abel-Jacobi map.

The first one is the familiar 
\[
\fname{Quot}_{(\mathcal{F}/X/S)}:\catname{Sch}_S\rightarrow\catname{Set}
\]

which maps an $S$-scheme $T$ to the set of locally finitely presented $T$-flat quotients of $\mathcal{F}_T$ whose support is proper and finitely presented over $T$.

In addition, if $\phi$ is a polynomial in $\QQ[x]$, we define 
\[
\fname{Quot}^{\phi}_{(\mathcal{F}/X/S)}\subset\fname{Quot}_{(\mathcal{F}/X/S)}
\]

to be the set of such quotients with Hilbert polynomial $\phi$ on each fiber.

\begin{defn}\cite[2.5]{AK}
Let $X\xrightarrow{f}S$ be a finitely presented morphism of schemes, $\mathcal{F}$ a locally finitely presented $\OO_X$-module, and $\mathcal{G}$ an $S$-flat quotient of $\mathcal{F}$. Define the \emph{pseudo-ideal} $\mathcal{I}(\mathcal{G})$ as the kernel of $\mathcal{F}\rightarrow\mathcal{G}\rightarrow 0$. 
\end{defn}

\begin{rem}
We require $\mathcal{G}$ to be flat so that the formation of the pseudo-ideal commutes with base change.
\end{rem}

Given a line bundle $\mathcal{L}$ on a curve, we have the notion of a linear system which is defined to be the collection of effective divisors $D$ such that $\OO(D)$ is isomorphic to $\mathcal{L}$. 
In \cite[4.1]{AK}, Altman and Kleiman generalize this as follows. Let $X\xrightarrow{f}S$ and $\mathcal{F}$ be as before, and let $\mathcal{I}$ be a finitely presented $\OO_X$-module. We define the functor
\[
\fname{LinSyst}_{(\mathcal{I},\mathcal{F})}\subset\fname{Quot}_{(\mathcal{F}/X/S)}
\]

to be the functor which maps an $S$-scheme $T$ to the set of quotients $\mathcal{G}\in\fname{Quot}_{(\mathcal{F}/X/S)}$ such that there exists a line bundle $\mathcal{N}$ on $T$ and an isomorphism
\[
\mathcal{I}(\mathcal{G})\cong\mathcal{I}\otimes\mathcal{N}
\]

One of the key observations in \cite{AK} is that the linear systems functor is representable by a very natural projective scheme on $S$.
This is precisely analogous to the fact that the linear system of a line bundle on a curve can be described as the projective space associated to the module of the line bundle's global sections.

In fact, the same thing is true in this case when $\mathcal{F}=\OO_X$ and $\mathcal{I}$ is an ideal of $\OO_X$. For arbitrary $\mathcal{F}$ and $\mathcal{I}$, we generalize the notion of global sections via the following combination of lemma and definition.

\begin{lem}(\cite[1.1]{AK})\label{lem-HIF}
Let $X\xrightarrow{f}S$ be a finitely presented morphism of schemes and let $\mathcal{I}$ and $\mathcal{F}$ be two locally finitely presented $\OO_X$-modules with $\mathcal{F}$ flat over $S$.
Then, there exists an locally finitely presented $\OO_X$-module $H(\mathcal{I},\mathcal{F})$ and an element $h(\mathcal{I},\mathcal{F})$ of $\Hom_X(\mathcal{I},\mathcal{F}\otimes H(\mathcal{I},\mathcal{F}))$ which together represent the covariant functor:
\[
\mathcal{M} \mapsto \Hom_X(\mathcal{I},\mathcal{F}\otimes \mathcal{M})
\]

from the category of quasi-coherent $\OO_S$-modules to $\catname{Set}$.
\end{lem}

\begin{rem}
It's easy to see that $H(\mathcal{I},\mathcal{J})$ is covariant in both variables.
\end{rem}

We can now rigorously state the observation we mentioned above.

\begin{thm}\cite[4.2]{AK}\label{thm-linsyst-representable}
Let $X\xrightarrow{f}S$ be a proper finitely presented morphism of schemes and let $\mathcal{F}$ and $\mathcal{I}$ be two finitely presented $\OO_X$-modules. Assume that $\mathcal{F}$ is $S$-flat and that, for each $S$-scheme $T$ for which $\mathcal{I}_T$ is $T$-flat, the canonical map
\[
\OO^{\times}_T\rightarrow (f_T)_*\fname{Isom}_{X_T}(\mathcal{I}_T,\mathcal{I}_T)
\]
is an isomorphism.

Then, the functor $\fname{LinSyst}_{(\mathcal{I},\mathcal{F})}$ is representable by an open subscheme $U$ of $\PP(H(\mathcal{I},\mathcal{F}))$ such that the inclusion is quasi-compact. 
Moreover, $U$ is isomorphic to $\PP(H(\mathcal{I},\mathcal{F}))$ if and only if, for each geometric point $s$ of $S$, the map $\mathcal{I}(s)\rightarrow\mathcal{F}(s)$ is an injection.
\end{thm}

The proof of this theorem is purely formal, and involves writing down the natural functor represented $\PP(H(\mathcal{I},\mathcal{F}))$ and unraveling the definition of $H(\mathcal{I},\mathcal{F})$.

In light of this role played by $H(\mathcal{I},\mathcal{F})$, the following theorem is crucial.

\begin{thm}\cite[1.3]{AK}\label{thm-free-HIJ}
Let $X\xrightarrow{f}S$ be a finitely presented, proper morphism of schemes, and let $\mathcal{I}$ and $\mathcal{F}$ be locally finitely presented $S$-flat $\OO_X$-modules. Assume,
\[
\Ext^1_{X(s)}(\mathcal{I}(s),\mathcal{F}(s)) = 0
\]
for some point $s$ in $S$.
Then there exists an open neighborhood $U$ of $s$ such that $H(\mathcal{I},\mathcal{F})|_U$ is locally free with finite rank. 
\end{thm}

This tells us that in good conditions, the linear systems functor is representable by a projective bundle.

In order to define the Abel Jacobi map, we'll need to be more selective in the types of quotients that we allow.
This is done by requiring that the pseudo-ideal of our quotients satisfy the following condition.

\begin{defn}\cite[5.1]{AK}
Let $X\xrightarrow{f}S$ be a morphism of schemes and let $\mathcal{I}$ be an $\OO_X$-module. Then $\mathcal{I}$ will be called \emph{simple over} $S$ or $S$-\emph{simple} if $\mathcal{I}$ is locally finitely presented and flat over $S$ and if for every $S$ scheme $T$, the canonical map
\[
\OO_T\rightarrow (f_{T})_*:\fname{Hom}_{X_T}(\mathcal{I}_T,\mathcal{I}_T)
\]
is an isomorphism.
\end{defn}

The following stronger property will also play a central role.
 
\begin{defn}\cite[3.1]{AK}
Let $X$ be a geometrically integral scheme over a field $k$ and let $\mathcal{I}$ be a coherent $\OO_X$ module.
We call $\mathcal{I}$ a \emph{rank-1, torsion-free sheaf} if it is torsion free and generically isomorphic to $\OO_X$.
\end{defn}

There is also the relative version.

\begin{defn}\cite[5.10]{AK}
Let $X\xrightarrow{f}S$ be a finitely presented proper morphism with integral geometric fibers. 
We call an $\OO_X$-module $\mathcal{I}$ \emph{relatively rank-1, torsion-free} over $S$ if it is locally finitely presented and flat over $S$ and if for each geometric point $s$ in $S$, the fiber $\mathcal{I}(s)$ is a rank-1, torsion-free $\OO_{X_s}$-module.
\end{defn}

We can now define the functor which is the target of the Abel-Jacobi map.

\begin{defn}\cite[5.5]{AK}
Let $X\xrightarrow{f}S$ be a morphism of schemes. We define the functor
\[
\fname{Spl}_{(X/S)}:\catname{Sch}_S\rightarrow\catname{Set}
\]
to be the functor which assigns to each $S$-scheme $T$, the set of equivalence classes of simple $\OO_{X_T}$-modules, where two modules $\mathcal{I}$ and $\mathcal{J}$ are defined to be equivalent if there exists a line bundle $\mathcal{N}$ on $T$ such that
\[
\mathcal{I}\otimes N \cong \mathcal{J}
\]

Similarly, we define the \emph{compactified Picard functor} to be the subfunctor 
\[
\overline{\fname{Pic}}_{(X/S)} \subset \fname{Spl}_{(X/S)}
\]
which maps an $S$-scheme $T$ to the classes in $\fname{Spl}_{(X/S)}(T)$ which are represented by a relatively rank-1 torsion-free $\OO_{X_T}$-module.
\end{defn}

As usual, we define $\fname{Spl}_{(X/S)(\textrm{\'{e}t})}$ and $\overline{\fname{Pic}}_{(X/S)(\textrm{\'{e}t})}$ to be the associated sheaves in the \'{e}tale topology. However, this distinction will not play a role in our application.
Just like with the $\fname{Quot}$ functor, after fixing a very ample line bundle on $X$, for a given polynomial $\phi$ we denote by $\fname{Spl}^{\phi}_{(X/S)(\textrm{\'{e}t})}$ and $\overline{\fname{Pic}}^{\phi}_{(X/S)(\textrm{\'{e}t})}$ the subfunctors defined by the additional condition that the $\OO_{X_T}$ module $\mathcal{I}$ have a Hilbert polynomial $\phi$.
 
Similarly, the source of the Abel-Jacobi map is defined as follow.

\begin{defn}\cite[5.14]{AK}
Let $X\xrightarrow{f}S$ be a finitely presented proper morphism of schemes and let $\mathcal{F}$ be a locally finitely presented $\OO_X$-module.
We define the functor
\[
\fname{Smp}_{(\mathcal{F}/X/S)} \subset \fname{Quot}_{(\mathcal{F}/X/S)}
\]

to be the functor which assigns to each $S$-scheme $T$, the quotients $\mathcal{G}\in\fname{Quot}_{(\mathcal{F}/X/S)}$ whose pseudo-ideals $\mathcal{I}(\mathcal{G})$ are simple over $S$.
\end{defn}

Finally, we can define the Abel-Jacobi map.

\begin{defn}\cite[5.16]{AK}
Let $X\xrightarrow{f}S$ be a proper, finitely presented morphism of schemes, and let $\mathcal{F}$ be a locally finitely presented $\mathcal{O}_X$-module. We define the \emph{Abel-Jacobi map associated to} $\mathcal{F}$ to be the map of functors:
\[
\mathcal{A}_{(\mathcal{F}/X/S)}:\fname{Smp}_{(\mathcal{F}/X/S)}\rightarrow \fname{Spl}_{(X/S)(\textrm{\'{e}t})}
\]

which takes a quotient $\mathcal{G}$ of $\mathcal{F}$ to the equivalence class of its pseudo-ideal $\mathcal{I}(\mathcal{G})$.
\end{defn}

Note that when $X$ is a smooth curve over a field and $\mathcal{F}=\OO_X$, this is the standard Abel-Jacobi map of a curve. 
As in that case, the fibers of the Abel-Jacobi map are naturally linear systems in the following sense.

\begin{lem}\cite[5.17]{AK}\label{lem-AJ-fibers}
Let $X\xrightarrow{f}S$ be a proper, finitely presented morphism of schemes with integral geometric fibers, and let $\mathcal{F}$ be an $S$-flat, locally finitely presented $\OO_X$-module. Let $T$ be an $S$-scheme and let $\mathcal{I}$ be a $T$-simple $\OO_{X_T}$-module. 
Furthermore, suppose that the geometric fibers of $f$ are integral and that $\mathcal{I}$ and $\mathcal{F}$ are relatively rank-1 torsion-free.
Then we have the following commutative diagram with a right Cartesian square:
\[
\xymatrix{
\PP(H(\mathcal{I},\mathcal{F}_T)) \ar[dr] \ar[r]^{\cong}   &   \fname{LinSyst}_{(\mathcal{I},\mathcal{F}_T)} \ar[d] \ar[r]  &  \fname{Smp}_{(\mathcal{F}/X/S)} \ar[d]^{\mathcal{A}_{\mathcal{F}}}  \\
                                                         &   T  \ar[r]^{\tau_{\mathcal{I}}}                      &  \fname{Spl}_{(X/S)(\textrm{\'{e}t})}
}
\]
where $\PP(H(\mathcal{I},\mathcal{F}_T))$ and $T$ stand for the respective functors of points and $\tau_{\mathcal{I}}$ is the map of functors taking a $T$-scheme $Y$ to the $Y$-simple ideal $\mathcal{I}_Y$ on $Y\times_SX$. 

In particular, for every geometric point $t$ of $ \fname{Spl}_{(X/S)(\textrm{\'{e}t})}$, if $\mathcal{I}$ is a representing $\OO_{X_t}$-module then
\[
\dim(\mathcal{A}^{-1}_{\mathcal{F}}) = \dim_{k(t)}(\Hom_{X(t)}(H(\mathcal{I},\mathcal{F}), k(t))) -1 = \dim_{k(t)}(\Hom_{X(t)}(\mathcal{I},\mathcal{F})) -1
\]
\end{lem}

This follows almost immediately from theorem \ref{thm-linsyst-representable}.

The following theorem is a fairly straightforward corollary as well.

\begin{thm}\cite[5.20-5.21]{AK}\label{thm-AJ-general-properties}
Let $X\xrightarrow{f}S$ be a finitely presented, proper morphism of schemes, whose geometric fibers are integral. Let $\mathcal{F}$ be a relatively rank-1 torsion-free $\OO_X$-module. Suppose that $P$ represents $\overline{\fname{Pic}}_{(X/S)(\textrm{\'{e}t})}$. Then:
\begin{enumerate}
\item The restriction of $\mathcal{A}_{\mathcal{F}}|_P$ is proper and finitely presented.
\item Assume that $f$ is projective and that the fibers $X(s)$ (resp. $\mathcal{F}(s)$) all have the same Hilbert polynomial $\phi$ (resp. $\psi$).
Then for each polynomial $\theta$, the restriction $\mathcal{A}|_{P^{\theta}}$ is strongly projective. This means that it factors as a closed embedding into a projective bundle over $P^{\theta}$, followed by the projection.
\item Assume there exists a universal $\OO_{X_P}$-module $\mathcal{I}$ such that $(P,\mathcal{I})$ represents $\overline{\fname{Pic}}_{(X/S)(\textrm{\'{e}t})}$. Then $\mathcal{A}_{\mathcal{F}}|_P$ is equal to the structure map of $\PP(H(\mathcal{I},\mathcal{F}_P))$ over $P$. Furthermore, this condition holds if the smooth locus of $X/S$ admits a section.
\end{enumerate}
\end{thm}

Note that condition in part 3 of the theorem implies in particular that $\overline{\fname{Pic}}_{(X/S)}$ is already an \'{e}tale sheaf and so $\overline{\fname{Pic}}_{(X/S)} = \overline{\fname{Pic}}_{(X/S)(\textrm{\'{e}t})}$

From the discussion up to this point we can see that if we forget the gnarly issue of representability, the structure of the Abel-Jacobi map as a natural transformation of functors is fairly transparent and requires minimal machinery.
Furthermore, nothing up to this point depended on special facts about curves.
Since our current application will only require access to this functorial description, using this setup will allow us to obtain cleaner statements and proofs.

We now narrow our focus to the case where $X\xrightarrow{f}S$ is a family of integral projective curves. 
In other words, $f$ is a flat, locally projective, finitely presented morphism of schemes whose geometric fibers are integral curves.

The advantage of this is that Riemann Roch now allows us to satisfy the conditions of theorem \ref{thm-free-HIJ} fairly easily.
Combined with theorem \ref{thm-AJ-general-properties}, this will force the Abel Jacobi map to be the structure map of a projective bundle in sufficiently general situations. 
Another benefit is that we now have the following representability result.

\begin{thm}\cite[8.1]{AK}
Let $X\xrightarrow{f}S$ be a locally projective, finitely presented, flat morphism of schemes whose geometric fibers are integral curves. Then, 
$\overline{\fname{Pic}}_{(X/S)(\textrm{\'{e}t})}$ is represented by a disjoint union of $S$-schemes $\coprod P_n$ where $P_n = \overline{\Pic}^n_{(X/S)(\textrm{\'{e}t})}$ and $P_n$ parametrizes the rank-1 torsion-free sheaves with Euler characteristic $n$ on the fibers of $X/S$.
\end{thm}

Furthermore, in this case, if $\mathcal{F}$ is a locally finitely presented $S$-flat $\OO_X$-module, then we have \cite[8.2.1]{AK}:

\[
\fname{Smp}_{(\mathcal{F}/X/S)} = \fname{Quot}_{(\mathcal{F}/X/S)}
\]
 
Therefore, in this case the restriction of the Abel-Jacobi map to the compactified Picard scheme takes the following more familiar form:
\[
\mathcal{A}_{\mathcal{F}}:\Quot_{(\mathcal{F}/X/S)} \rightarrow P = \overline{\Pic}_{(X/S)(\textrm{\'{e}t})}
\]
Where $\Quot_{(\mathcal{F}/X/S)}$ is the scheme representing the functor $\fname{Quot}_{(\mathcal{F}/X/S)}$.

In addition, if $\chi(\mathcal{F}(s))$ is independent of the point $s$ in $S$, the map decomposes as

\[
\mathcal{A}^d_{\mathcal{F}}:\Quot^d_{(\mathcal{F}/X/S)} \rightarrow P_n
\]
where $n = \chi(\mathcal{F}(s)) - d$. For example, this happens if $\mathcal{F}$ is the relative dualizing sheaf $\omega$ of $X/S$ and the fibers of $f$ have the same arithmetic genus.

The final result that we will need now follows fairly easily from Riemann-Roch combined with theorems \ref{thm-free-HIJ} and \ref{thm-AJ-general-properties} as we mentioned earlier.

\begin{thm}\cite[8.4]{AK}\label{thm-AJ-curve-properties}
Let $X\xrightarrow{f}S$ be a flat, finitely presented, locally projective morphism whose geometric fibers are integral curves with the same geometric genus $p$. Let $\omega$ denote the relative dualizing sheaf. Fix an integer $d\geq 2p-1$. Then the $d$-th part of the Abel-Jacobi map
\[
\mathcal{A}^d_{\omega}:\Quot^d_{(\omega/X/S)} \rightarrow P_{p-1-d}
\]
is smooth with relative dimension $d-p$.
\end{thm}

Note that for any line bundle $\mathcal{L}$ on $X$ and any integer $n$, tensoring with $\mathcal{L}$ defines an isomorphism
\[
\nu_{\mathcal{L}}: P_n\rightarrow P_m
\]
where $m = n+\chi(\mathcal{L})$

\subsubsection{The Riemann-Kempf Formula in the Smooth Case}\label{sec-RK-smooth-case}

As motivation for the general case, we will briefly recall the original statement of the Riemann singularity theorem, and the proof given in \cite[4.3.2]{Fulton}.

Let $X$ be a projective non-singular curve of genus $g$ over an algebraically closed field $k$. Let $X^{(n)}$ denote the $n$-th symmetric power of the curve.
Note that in this case
\[
X^{(n)}\cong \Hilb^d_X \cong \Quot^d_{(\omega/X)}
\]

and $\overline{\Pic}^0_{(X/S)(\textrm{\'{e}t})}$ is isomorphic to the Jacobian $J_X$.

Therefore, we have an Abel-Jacobi map for each integer $d$
\[
A^{d}:X^{(d)}\rightarrow J_X
\]
as defined in the previous section which (after taking into account the identification of $\Hilb^d_X$ and $\Quot^d_{(\omega/X)}$) takes a degree $d$ divisor $D$ to the line bundle $\OO(D)$.

Let $W_d$ denote the image of $A^d$. In particular, $W_{g-1}$ is the theta divisor.

\begin{thm}\emph{(Riemann-Kempf formula, smooth case)}
Let $D$ be a degree $d$ divisor. Then the multiplicity of $W_d$ at $A^d(D)$ is
\[
\binom{g-d+r}{r}
\]
where $r+1$ is the dimension of the linear system $|D|\cong\PP^r_k$.

In particular, when $d=g-1$, the multiplicity is $r$. This is known as the \emph{Riemann singularity theorem}.
\end{thm}

In \cite[4.3.2]{Fulton}, Fulton gives a beautiful proof of this theorem using elementary facts about Segre classes, together with the following properties of the Abel-Jacobi map which are well known (and implied by theorem \ref{thm-AJ-curve-properties} and lemma \ref{lem-AJ-fibers}):

\begin{enumerate}
\item The scheme theoretic fibers of $A^d$ are the linear systems $|D|\cong\PP^r_k$.
\item If $d\geq 2g-1$ then $A^d$ makes $X^{(d)}$ a projective bundle over $J_X$ of relative dimension $d-g$.
\item If $1\leq d \leq g$ then $A^d$ maps $X^{(d)}$ birationally onto it's image.
\end{enumerate}

Let $D$ be a degree $d$ divisor and $r$ an integer such that $|D|\cong\PP^r$. Fulton obtains the theorem as a corollary of the following more general fact:

\begin{equation}\label{eq-segre}
s(|D|,X^{(d)}) = (1+h)^{g-d+r}\cap [|D|]
\end{equation}
where $h$ is the first Chern class of the canonical line bundle on $|D|=\PP^r$. 
This implies the theorem since the degree of the pushforward of this Segre class is the multiplicity of $A^d(|D|)$ in $W_d$, and the degree of $(1+h)^{g-d+r}\cap [\PP^r]$ is precisely:
\[
\binom{g-d+r}{r}
\]

The proof of equation \ref{eq-segre} is done by reduction to the case where $d>>0$.

Let $d$ be an arbitrary integer. 
In order to reduce to the case where $d$ is large, we choose some integer $s$ such that $d+s\geq 2g-1$.
Furthermore, by choosing a point $p$ in $X$, we obtain an embedding
\[
X^{(d)}\hookrightarrow X^{(d+s)}
\]
by sending a degree $d$ divisor $D$ to $D+sp$. 
The reduction from the degree $d+s$ case to the degree $d$ case is motivated by the following commutative diagram with a left (but not right) fibered square:

\[
\xymatrix{
|D| \ar[r] \ar[d]   &  X^{(d)} \ar[d]^{A^d} \ar[r]   &  X^{(d+s)}  \ar[d]^{A^{d+s}}  \\
A^d(|D|)  \ar[r]    &  W_d \ar[r]                  &  J_X
}
\]

\begin{description}
\item[Proof when $d$ is large:] \hfill \\

When $d\geq 2g-1$, then $A^d$ is a smooth map and in particular it is flat. Therefore, since $J_X$ is smooth,

\begin{equation}\label{eq-segre-large-degree}
s(|D|,C^{(d)}) = (A^d)^*s(A^d(D), J_X) = (A^d)^*(\mult_{A^d}\cdot J_X[A^d(D)]) = \mult_{A^d}\cdot J_X[|D|] = [|D|]
\end{equation}

\item[Calculation of $s(|D|,X^{(d+s)})$ using $s(|D+sp|, X^{(d+s)})$ when $d+s$ is large:] \hfill \\

By equation \ref{eq-segre-large-degree}, we know that
\[
s(|D+sp|, X^{(d+s)}) = [|D+sp|] = [\PP^{d+s-g}]
\]

Since $|D|=\PP^r$, we have the following composition of closed embeddings
\[
|D|=\PP^r\rightarrow\PP^{d+s-g}\rightarrow X^{(d+s)}
\]
and since everything is smooth, from the conormal sequence we can deduce that
\begin{equation}\label{eq-segre-comparing-degrees}
\xymatrix @ R=0.3pc {
[|D|] = s(|D+sp|,X^{(d+s)})|_{|D|} = c(N_{|D|}|D+sp|)\cap s(|D|,X^{(d+s)}) \\
= (1+h)^{d+s-g-r}\cap s(|D|,X^{(d+s)})
}
\end{equation}
where $h$ is the first Chern class of the canonical bundle on $|D|\cong \PP^r$
\item[Calculation of $s(|D|,X^{(d)})$ using $s(|D|, X^{(d+s)})$ when $d+s$ is large:] \hfill \\

Finally, we consider the sequence of embeddings
\[
|D|\rightarrow X^{(d)} \rightarrow X^{(d+s)}
\]
It is easy to check that $c(N_{X^{(d)}}X^{(d+s)}|_{|D|})=(1+h)^s$.
Therefore, by a version of cancellation of Segre classes which can be explicitly shown to hold in this case (or even more easily, by proposition \ref{prop-smooth-case}),
\begin{equation}\label{eq-segre-cancellation-divisors}
s(|D|,X^{(d)}) = c(N_{X^{(d)}}X^{(d+s)}|_{|D|})\cap s(|D|,X^{(d+s)}) = \frac{(1+h)^s}{(1+h)^{d+s-g-r}}\cap [|D|] = (1+h)^{g-d+r}\cap [|D|]
\end{equation}

\end{description}

Notice that other than the three facts about the Abel-Jacobi map that we stated above, each of the three steps involved manipulations of Segre classes which hold in far greater generality. Since the facts about the Abel-Jacobi map are still true in the situation of an integral projective curve from the previous section, it is tempting to think that generalizing Fulton's proof to that case can be reduced to verifying that these three steps are still valid. Indeed, in section \ref{sec-statements-proofs} we will see that the first two steps can be carried out in a nearly identical fashion for arbitrary integral curves. The third step will turn out to be valid as well, and can be justified by theorem \ref{thm-main-theorem}.

\subsubsection{The Riemann Singularity Theorem in the Nodal Case}
In \cite{CMK}, S. Casalania-Martin and J. Kass proved the following version of the Riemann singularity theorem for planar curves.

\begin{thm}\cite[A]{CMK}\label{thm-RS-nodal-case}
Let $X$ be an integral projective curve with at most planar singularities over an algebraically closed field $k$.
Let $x$ be a point on the theta divisor $\Theta$ corresponding to a rank-1, torsion-free sheaf $\mathcal{I}$.
In the sheaf $\mathcal{I}$ fails to be locally free at $n$ nodes and no other points, then the multiplicity and order of vanishing of $\Theta$ at $x$ satisfy the equation
\[
\mult_x\Theta = (\mult_x\overline{\Pic}_{X/k}) \cdot \ord_x\Theta = 2^n \cdot h^0(X,\mathcal{I}) 
\] 

Specifically,
\[
(\mult_x\overline{\Pic}_{X/k}) = 2^n
\]
and
\[
\mult_x\Theta = (\mult_x\overline{\Pic}_{X/k}) \cdot h^0(X,\mathcal{I}) 
\]
\end{thm}

It is interesting to note that in equation \ref{eq-segre-large-degree}, section \ref{sec-RK-smooth-case}, we dropped the term $\mult_{A^d}J_X$ since in the smooth case it is equal to $1$. However, if we would have held on to it, our derivation of the Riemann singularity theorem in the smooth case would have led us to an equality which symbolically in precisely theorem \ref{thm-RS-nodal-case}.

The proof of this theorem in \cite{CMK} used different methods, so it would be interesting to test the boundaries of Fulton's proof from section \ref{sec-RK-smooth-case}. In the next section we'll use our theorem for cancellation of Segre classes to show that this theorem holds for integral projective curves with planar singularities. In addition, we'll prove a version of the Riemann-Kempf formula for arbitrary families of integral curves.

\subsection{Statements and Proofs}\label{sec-statements-proofs}
\subsubsection{Statement and Proof of the Riemann Kempf Formula for Integral Curves}
In this section by proving a generalization of the Riemann-Kempf formula for integral curves. More specifically, we will generalize equation \ref{eq-segre}.
The notation here is taken from section \ref{sec-AK-reference}. In particular, $P_d$ will denote $\overline{\Pic}^d_{(X/S)(\textrm{\'{e}t})}$ and 
\[
\mathcal{A}^d_{\omega}:\Quot^d_{(\omega/X/S)} \rightarrow P_{p-1-d}
\]
is the $d$-th part of the Abel-Jacobi map.

The proof of this theorem will use a collection of technical results about Quot schemes which we will prove in section \ref{sec-technical-lemmas}.

\begin{thm}\label{thm-RKF-integral-curves}
Let $k$ be an algebraically closed field of characteristic $0$ and let $X/k$ an integral projective $k$-scheme with arithmetic genus $p$.

Let $d$ be an integer, $Z$ be an equi-dimensional subscheme of $P_{p-1-d}$, $\Quot^d_{(\omega/X/k)Z}$ the pullback of $\Quot^d_{(\omega/X/k)}$ and $\mathcal{A}^d_{\omega Z}$ the pullback of $\mathcal{A}^d_{\omega}$. 

Let $x$ be a $k$ point of $Z$ and $r$ an integer such that $(\mathcal{A}^d_{\omega Z})^{-1}(x)\cong \PP^r_k$. Then
\[
s((\mathcal{A}^d_{\omega Z})^{-1}(x), \Quot^d_{(\omega/X/k)Z}) = \mult_xZ \cdot (1+h)^{p-d+r}\cap [(\mathcal{A}^d_{\omega Z})^{-1}(x)]
\]
where $h$ is the first Chern class of the canonical bundle on $\PP^r_k$.
\end{thm}

\begin{rem}
The condition that $Z$ be equi-dimensional comes from the conditions of proposition \cite[4.2]{Fulton}.
In general, proper pushforward of segre classes in ill defined for schemes with components of varying dimensions.
\end{rem}

\begin{cor}\label{cor-RKF-planar-case}
Let $k$ be a algebraically closed field of characteristic $0$ and let $X$ be a projective integral curve of arithmetic genus $p$ over $k$ with at most planar singularities.
Let $x$ be a $k$-point of $P_{p-1-d}$ and $r$ an integer such that $(\mathcal{A}^d_{\omega})^{-1}(x)\cong \PP^r_k$.
Then
\[
s((\mathcal{A}^d_{\omega})^{-1}(p), \Quot^d_{(\omega/X/k)}) = \mult_xP_{p-1-d} \cdot (1+h)^{p-d+r}\cap [(\mathcal{A}^d_{\omega})^{-1}(p)]
\]
where $h$ is the first Chern class of the canonical bundle on $\PP^r_k$.
\end{cor}
\begin{proof}
By \cite[9]{AIK}, $P_{p-1-d}$ is irreducible in this case. 
The corollary then follows immediately from theorem \ref{thm-RKF-integral-curves}.
\end{proof}

\begin{cor}\label{cor-RST-planar-case}
Let $k$ be a algebraically closed field of characteristic $0$ and let $X$ be a projective integral curve of arithmetic genus $p$ over $k$ with at most planar singularities.
Let $x$ be a $k$-point of $P_0$ corresponding to an rank-1 torsion free sheaf $\mathcal{I}$.
Let $\Theta$ denote the image of $\mathcal{A}^{p-1}_{\omega}$ in $P_0$. Then
\[
\mult_x\Theta = \mult_xP_{0} \cdot (h^1(X,\mathcal{I}) - 1)
\]
where $h$ is the first Chern class of the canonical bundle on $(\mathcal{A}^{p-1}_{\omega})^{-1}(x)\cong\PP^r_k$ and $r = h^1(X,\mathcal{I}) - 1$.
\end{cor}
\begin{proof}
By \cite[3.1]{CMK}, $\Theta$ is irreducible. Furthermore, by looking at the dimensions of the fibers, $\mathcal{A}^{p-1}_{\omega}$ is birational onto its image. The corollary now follows since
\[
s(x,\Theta) = (\mathcal{A}^{p-1}_{\omega})_{*}s((\mathcal{A}^d_{\omega})^{-1}(p), \Quot^d_{(\omega/X/S)})
\]
and by corollary \ref{cor-RKF-planar-case} this is equal to $\mult_xP_{0} \cdot (h^1(X,\mathcal{I}) - 1)$
\end{proof}

We now turn to the proof of theorem \ref{thm-RKF-integral-curves}. 
We will follow the steps of Fulton's proof from section \ref{sec-RK-smooth-case}, and show that everything works in this more general setting.

As in that proof, the key step will be to reduce the general case to case where $d$ is large.
To this end, for a fixed $d$, choose $s$ such that $d+s > 2g -1$.
For each $i$ and $1\leq j \leq s$ we will define natural maps of functors
\[
\fname{Quot}^{i}_{(\omega/X/k)} \xrightarrow{q_{j}} \fname{Quot}^{i+1}_{(\omega/X/k)}
\]
. The maps $q_j$ all increase the degree by one so the degree $i$ will be implied by the context. We will then show that these maps have properties which allow us to carry out the reductions that we used in Fulton's proof of the smooth case.

Let $x_1,\dots,x_{s}$ be distinct $k$-points in the smooth locus of $X$ and let $\mathcal{I}_{x_i}$ denote their ideal sheaves. For each $i$ and $1\leq j \leq s$ we define the map
\[
\fname{Quot}^{i}_{(\omega/X/k)} \xrightarrow{q_{j}} \fname{Quot}^{i+1}_{(\omega/X/k)}
\]
as follows.
Let $T$ be an $S$ scheme. Then our map will send a quotient
\[
0\rightarrow \mathcal{I}(G) \rightarrow \omega \rightarrow \mathcal{G} \rightarrow 0
\]
in $\fname{Quot}^{i}_{(\omega/X/k)}(T)$ to the quotient
\[
0\rightarrow \mathcal{I}(G)\otimes\mathcal{I}_{x_j} \rightarrow \omega \rightarrow \mathcal{G}' \rightarrow 0
\]
in $\fname{Quot}^{i+1}_{(\omega/X/k)}(T)$ where $\mathcal{G}'$ is defined to be the quotient. 
One way to verify that this is in fact a morphism of functors is to note that $X$ has an affine cover 
\[
X = (X\setminus x_j) \coprod X^{\textrm{sm}}
\]
such that on the first chart $\mathcal{I}_{x_j}$ is trivial, and on the second chart both $\mathcal{I}_{x_j}$ and $\omega$ are line bundles.

\begin{rem}\label{rem-intersection-of-embeddings}
These morphisms are related by the following fiber diagram
\[
\xymatrix{
\fname{Quot}^{i}_{(\omega/X/k)} \ar[r]^{q_{l}} \ar[d]^{q_j} & \fname{Quot}^{i+1}_{(\omega/X/k)} \ar[d]^{q_j} \\
\fname{Quot}^{i+1}_{(\omega/X/k)} \ar[r]^{q_l}               & \fname{Quot}^{i+2}_{(\omega/X/k)}
}
\]
for any $i$ and $1\leq j,k \leq s$.

As we'll soon see, the maps $q_j$ are closed embeddings so at the level of the schemes representing these functors, this implies that composing the maps $q_l$ and $q_j$ is equivalent to taking the intersection of the images of these maps. Additional properties of the maps $q_j$ will be proved in section \ref{sec-technical-lemmas}.
\end{rem}

\begin{proof}[Proof of theorem \ref{thm-RKF-integral-curves}:] \hfill \\
\begin{description}
\item[Proof when $d$ is large:] \hfill \\

When $d\geq 2g-1$, then by theorem \ref{thm-AJ-curve-properties}, $\mathcal{A}^d_{\omega}$ is a smooth map and in particular it is flat. Therefore, 
\[
s((\mathcal{A}^d_{\omega Z})^{-1}(x), \Quot^d_{(\omega/X/k)Z}) =  (\mathcal{A}^d_{\omega Z})^*s(x, Z) = \mult_xZ \cdot [(\mathcal{A}^d_{\omega Z})^{-1}(x)]
\]

\item[Calculation of $s((\mathcal{A}^d_{\omega Z})^{-1}(x),\Quot^{d+s}_{(\omega/X/k)Z})$ using $s((\mathcal{A}^{d+s}_{\omega Z})^{-1}(x), \Quot^{d+s}_{(\omega/X/k)Z})$ when $d+s$ is large:] \hfill \\
By lemma \ref{lem-AJ-fibers}, we know that
\[
(\mathcal{A}^d_{\omega Z})^{-1}(x) \cong \PP^r_k
\]
for some $r$ and together with theorem \ref{thm-AJ-curve-properties} we know that 
\[
(\mathcal{A}^{d+s}_{\omega Z})^{-1}(x) \cong \PP^{d+s-p}_k
\]
For the rest of this step we will identify the fibers with these projective spaces.
By the previous step, we know that
\[
s(\PP^{d+s-p}_k, \Quot^{d+s-p}_{(\omega/X/k)Z}) = \mult_xZ \cdot [\PP^{d+s-p}_k]
\]

Furthermore, we have the following composition of closed embeddings
\[
\PP^r\xrightarrow{i}\PP^{d+s-p}\rightarrow X^{(d+s)}
\]
and by lemma \ref{lem-induced-embedding-of-fibers} the first embedding is regular and the Chern class of its normal bundle is
\[
(1+h)^{d+s-p-r}
\]
where $h$ is the first Chern class of the canonical line bundle on $\PP^r_k$.
Therefore, by lemma \ref{lem-composition-into-smooth}, we have
\[
\xymatrix @ R=0.3pc { 
\mult_xZ \cdot [\PP^r] = i^*s(\PP^{d+s-p}_k,\Quot^{d+s-p}_{(\omega/X/S)Z}) \\
= (1+h)^{d+s-p-r}\cap s(\PP^r_k,\Quot^{d+s}_{(\omega/X/S)Z})
}
\]

\item[Calculation of $s((\mathcal{A}^d_{\omega Z})^{-1}(x),\Quot^d_{(\omega/X/k)Z})$ using $s((\mathcal{A}^d_{\omega Z})^{-1}(x),\Quot^{d+s}_{(\omega/X/k)Z})$ when $d+s$ is large:] \hfill \\
As before, we have
\[
(\mathcal{A}^d_{\omega Z})^{-1}(x) \cong \PP^r_k
\]

Consider the sequence of embeddings
\[
\PP^r_k \rightarrow \Quot^{d}_{(\omega/X/k)Z} \xrightarrow{q} \Quot^{d+s}_{(\omega/X/k)Z}
\]
where $q$ is equal to the composition $q_1\circ q_2 \dots \circ q_s$.
By lemma \ref{lem-embedding-of-quot-is-cartier} combined with remark \ref{rem-intersection-of-embeddings}, the restriction of the normal bundle of the embedding $q$ to $\PP^r_k$ is $(1+h)^s$.

Therefore, by lemma \ref{lem-local-quot-structure} and our cancellation theorem for Segre classes (theorem \ref{thm-main-theorem})

\[
\xymatrix @ R=0.3pc {
s(\PP^r_k,\Quot^{d}_{(\omega/X/k)Z}) = (1+h)^s \cap s(\PP^r,\Quot^{d+s}_{(\omega/X/k)Z}) = \\
\mult_xZ \cdot \frac{(1+h)^s}{(1+h)^{d+s-p-r}}\cap [\PP^r_k] = \mult_xZ \cdot (1+h)^{p-d+r}\cap [\PP^r_k]
}
\]

\end{description}
\end{proof}

\subsubsection{A Collection of Technical Lemmas}\label{sec-technical-lemmas}
In order to use the maps $q_j$, we will need to prove that they are sufficiently nice as formalized by the lemmas in this section. Many of the results here are well known, but are stated for convenience.

Before doing this, we will cite two well-known lemmas (some people will call these definitions) from \cite{S} which ,incidentally, are used there to prove analogous properties of these maps in the smooth case.

\begin{lem}(\cite[3]{S})\label{lem-epimorphism-of-sheaves}
An epimorphism of coherent sheaves $u:\mathcal{E}\rightarrow\mathcal{F}$ on a variety $J$ induces a closed immersion
\[
\PP(\mathcal{F})\xrightarrow{q}\PP(\mathcal{E})
\]
such that $u^*\OO_{\PP(\mathcal{E})}(1)\cong \OO_{\PP(\mathcal{F})}(1)$
\end{lem}

\begin{lem}(\cite[4]{S})\label{lem-epimorphism-of-sheaves-good-kernel}
Let $u:\mathcal{E}\rightarrow\mathcal{F}$ be an epimorphism of coherent sheaves on a variety $J$ with kernel $\OO_J$. Then the induced closed immersion $\PP(\mathcal{F})\xrightarrow{q}\PP(\mathcal{E})$ is represented by the sheaf $u^*\OO_{\PP(\mathcal{E})}(-1)$.
\end{lem}

We will also need the following fact from deformation theory.
\begin{lem}(\cite[Ex. 14.3]{H})\label{lem-local-deformations}
Let $k$ be an algebraically closed field, let $\mathcal{F}$ be a functor
\[
\mathcal{F}:\catname{Sch}_k\rightarrow \catname{Set}
\]
and let $X_0$ be an element of $\mathcal{F}(\Spec{k})$.
We can define a local functor
\[
\fname{F}:\catname{C} \rightarrow \catname{Set}
\]
where $\catname{C}$ is the category of finitely generated local Artin rings over $k$ with residue field $k$, by sending an element $A\in\catname{C}$ to the subset of $\mathcal{F}(\Spec{A})$ consisting of those elements $X\in\mathcal{F}(\Spec{A})$ that reduce to $X_0\in\mathcal{F}(\Spec{k})$ under the natural pull-back morphism. We call this the \emph{functor of local deformations}.

If $\mathcal{F}$ is representable by a scheme $M$ and a family $\mathcal{I}\in\mathcal{F}(M)$ such that $X_0$ corresponds to the point $x_0$ in $M$, then $\fname{F}$ is pro-representable by the complete local ring $\hat{\OO}_{M,x_0}$
\end{lem}

All of the terms in this lemma are defined in $\cite[cpt. 14]{H}$.

We will now state and prove the three properties of our embeddings of $\fname{Quot}$-schemes that will be needed in our proof of theorem \ref{thm-RKF-integral-curves}.

\begin{lem}\label{lem-local-quot-structure}
For any integer $i$ and $1\leq j \leq s$, the morphism $\Quot^{i}_{(\omega/X/k)} \xrightarrow{q_j} \Quot^{i+1}_{(\omega/X/k)}$ induced by the corresponding map of functors is formally locally a section of a trivial projection in the sense of theorem \ref{thm-main-theorem}.
\end{lem}
\begin{proof}
To facilitate the notation, we will denote $\Quot^{l}_{(\omega/X/k)}$ by $\Q^{l}$.

Let $x$ be a $k$-point of $\Q^{i}$ corresponding to a quotient with kernel $\mathcal{J}_0$, which is mapped by $q_j$ to $y$ which represents a quotient with kernel $\mathcal{J}_0\otimes\mathcal{I}_{x_j}$. 
Then we have to prove that there exists an isomorphism $\varphi$ making the following diagram commute
\[
\xymatrix @ R=0.5pc {
  &  \Spec{\hat{\OO}_{\Q^{i+1},y}} \ar[dd]^{\varphi}_{\cong} \\
\Spec{\hat{\OO}_{\Q^{i},x}} \ar[ru]^{\hat{q_j}} \ar[rd]^s & \\
  & \Spec{\hat{\OO}_{\Q^{i},x} \PS}
}
\]
where $s$ is the zero section.

By lemma \ref{lem-local-deformations}, $\Spec{\hat{\OO}_{\Q^{i},x}}$ pro-represents the local deformation functor $\fname{F}_i$ of $x$ at $\Q^{i}$.
Therefore, we have a universal rank-1 torsion-free sheaf $\mathcal{J}$ on $X\times\Spec{\hat{\OO}_{\Q^{i},x}}$ which restricts to $\mathcal{J}_0$ on the closed point.
Similarly, $\Spec{\hat{\OO}_{\Q^{i+1},x}}$ pro-represents the local deformation functor $\fname{F}_{i+1}$ of $y$ at $\Q^{i+1}$ and we have a universal rank-1 torsion-free sheaf $\mathcal{K}$ on $X\times\Spec{\hat{\OO}_{\Q^{i+1},x}}$ which restricts to $\mathcal{J}_0\otimes\mathcal{I}_{x_j}$ on the closed point.
Finally, since $x_j$ was chosen to be a smooth point of $X$, $\Spec{\hat{\OO}_{\Q^{d+i},x} \PS}$ pro-represents the functor $\fname{F}_i\times\fname{G}$ where $\fname{G}$ is the local deformation functor of $\mathcal{I}_{x_j}$ in $\Q^1$.

Therefore, the existence of $\varphi$ is equivalent to the existence of an isomorphism of local deformation functors $\psi$ making the following diagram commute.

\[
\xymatrix @ R=0.5pc {
  &  \fname{F}^{i+1} \ar[dd]^{\psi}_{\cong} \\
\fname{F}^i \ar[ru]^{\alpha} \ar[rd]^{\beta} & \\
  & \fname{F}^i\times \fname{G}
}
\]
where $\alpha$ and $\beta$ are defined compatibly with the previous diagram of scheme morphisms. 
For convenience, we'll define them explicitly. 
Let $A$ be an Artin ring in $\mathcal{C}$ and $\mathcal{I}$ a sheaf in $X\times\Spec{A}$ corresponding to an element on $\fname{F}_i(A)$.
Then
\[
\xymatrix @ R=0.5pc{
\alpha(\mathcal{I}) = & \mathcal{I}\otimes\mathcal{I}_{x_j}\in\fname{F}_{i+1}(A) \\
\beta(\mathcal{I}) =  & (\mathcal{I},\OO_A\otimes\mathcal{I}_{x_j})\in\fname{F}_i(A)\times\fname{G}(A)
}
\]

There is a natural morphism
\[
\fname{F}_i\times\fname{G}\rightarrow\fname{F}_{i+1}
\]
sending a pair of elements in $\fname{F}_i(A)\times\fname{G}(A)$ to their tensor product.
Therefore, the proof will be complete if we can construct an inverse $\psi$ of this map.

Since $\mathcal{K}$ is universal, constructing $\psi$ is equivalent to finding a deformation
\[
0\rightarrow\mathcal{L}\rightarrow\omega
\]
of $\mathcal{I}_0$ over $\Spec{\hat{\OO}_{\Q^{i+1},y}}$ and a deformation 
\[
0\rightarrow\mathcal{M}_j\rightarrow\omega
\]
of $\mathcal{I}_{x_j}$ such that $\mathcal{L}\otimes\mathcal{M}_j\cong\mathcal{K}$. 
Because in that case, we could define $\psi$ by sending $\mathcal{K}$ to the pair $(\mathcal{L},\,\mathcal{M}_j)$.

Now, we know that the projection
\[
Y = X\times\Spec{\hat{\OO}_{\Q^{i+1},y}} \rightarrow \Spec{\hat{\OO}_{\Q^{i+1},y}}
\]
is proper. Furthermore, the support of the cokernel $\mathcal{G}$ of $\mathcal{K}\rightarrow\omega$ is proper and flat over $\Spec{\hat{\OO}_{\Q^{d+i+1},y}}$ and its restriction to the closed point consists of $i+1$ points, one of which is $x_j$.
Therefore, by the theorem of formal functions applied to the global sections of this support, the support of $\mathcal{G}$ itself consists of $i+1$ connected components. Let $C_j$ denote the component containing $x_j$ and let $C$ denote the union of the other components.

We now define $\mathcal{L}$ to be the sheaf which is $\mathcal{K}$ on $Y\setminus C_j$ and trivial on $Y\setminus C$. 
Similarly, we define $\mathcal{M}_j$ to be the sheaf which is $\mathcal{K}$ on $Y\setminus C$ and trivial on $Y\setminus C_j$.
\end{proof}

\begin{lem}\label{lem-induced-embedding-of-fibers}
For each integer $i$ and $1\leq j \leq s$, let $\Quot^{i}_{(\omega/X/k)} \xrightarrow{q_j} \Quot^{i+1}_{(\omega/X/k)}$  be the morphism defined above.
Let $p_i=[\mathcal{I}]$ be a $k$-point of $P_{p-i-1}$ and let $p_{i+1}=[\mathcal{I}\otimes\mathcal{I}_{x_j}]$ denote the corresponding element of $P_{p-i-2}$.
Let
\[
\PP^{r_i}_k \cong (\mathcal{A}^{d+i}_{\omega})^{-1}(p_i) \xrightarrow{u} (\mathcal{A}^{d+i+1}_{\omega})^{-1}(p_{i+1}) \cong \PP^{r_{i+1}}_k
\]
denote the map on fibers induced by $q_j$ where $r_i$ and $r_{i+1}$ are defined to be the dimensions of the fibers.
Then, $u$ is a degree one embedding of projective spaces.
\end{lem}

We also need the following similar looking lemma. 
Let $\mathcal{J}_i$ denote the tautological rank-1 torsion-free sheaf on $X\times P_{p-i-1}$.
Recall that by theorem \ref{thm-free-HIJ} and theorem \ref{thm-AJ-general-properties}, for each $i$, $\Quot^{i}_{(\omega/X/k)}$ is a projective scheme over $P_{p-i-1}$ of the form $\PP(H(\mathcal{J}_{i},\omega))$ and the Abel Jacobi map is the structure map. Furthermore, by our choice of $s$, $H(\mathcal{I}_{d+s-1},\omega)$ and $H(\mathcal{I}_{d+s},\omega)$ are locally free.

\begin{lem}\label{lem-embedding-of-quot-is-cartier}
For each $1\leq j \leq s$, the map $\Quot^{d+s-1}_{(\omega/X/k)} \xrightarrow{q_j} \Quot^{d+s}_{(\omega/X/k)}$ embeds the projective bundle $\Quot^{d+s-1}_{(\omega/X/k)}\cong\PP(H(\mathcal{J}_{d+s-1},\omega))$ as a degree one Cartier divisor in the projective bundle $\Quot^{d+s}_{(\omega/X/k)}\cong\PP(H(\mathcal{J}_{d+s},\omega))$ and is cut out by a section of $\OO_{\PP(H(\mathcal{I}_{d+s},\omega))}(1)$.
\end{lem}

In interest of economy and clarity, we will prove both of these lemmas as special cases of the following general statement about linear systems.

\begin{lem}\label{lem-embedding-linear-systems}
Let $l$ be a positive integer and $1\leq j \leq s$. Let $T$ be a $P_{p-i-1}$ scheme and denote the induced rank-1 torsion-free sheaf on $X\times_k T$ by $\mathcal{I}$. Consider the map $\fname{Quot}^{l}_{(\omega/X/S)} \xrightarrow{q_j} \fname{Quot}^{l+1}_{(\omega/X/S)}$ and the induced map on the linear systems as shown in the following diagram of Cartesian squares:
\[
\xymatrix{
\PP(H(\mathcal{I}, \omega)) \ar[r]^{\cong} \ar[d]  & \fname{LinSyst}_{(\mathcal{I},\omega)}  \ar[r] \ar[d] & \fname{Quot}^{i}_{(\omega/X/S)} \ar[d]^{q_j} & \\
\PP(H(\mathcal{I}\otimes\mathcal{I}_{x_j}, \omega)) \ar[r]^{\cong}  & \fname{LinSyst}_{(\mathcal{I}\otimes\mathcal{I}_{x_j},\omega)}  \ar[r] \ar[d] & \fname{Quot}^{i+1}_{(\omega/X/S)} \ar[d] \ar[r]^{\cong} & \fname{Quot}^{i+1}_{(\omega/X/S)} \ar[d]^{\mathcal{A}^{i+1}_{\omega}}\\
            &  T \ar[r]^{\mu_{\mathcal{I}}}  & P_{p-i-1} \ar[r]^{\nu_{\mathcal{I}_{x_j}}}_{\cong} & P_{p-i-2}
}
\]
Then, the map on linear systems is induced by the canonical map
\[
H(\mathcal{I}\otimes\mathcal{I}_{x_j}, \omega) \rightarrow H(\mathcal{I}, \omega)
\]
coming from the natural map $\mathcal{I}\otimes\mathcal{I}_{x_j}\rightarrow\mathcal{I}$, and this map is a surjection.
\end{lem}

\begin{proof}
The first statement follows from unraveling the identification of linear systems with projective space in lemma \ref{lem-AJ-fibers}. Since this is done in detail in the proof of this lemma in $\cite[5.17]{AIK}$, we won't reproduce it here.
For the second statement, let $\mathcal{G}$ be an $\OO_T$-module and let
\[
\xymatrix{
H(\mathcal{I}\otimes\mathcal{I}_{x_j}, \omega) \ar[r]^f & H(\mathcal{I}, \omega) \ar@<-.5ex>[r]_{\alpha} \ar@<.5ex>[r]^{\beta} & \mathcal{G}
}
\]
be maps such that $\alpha\circ f = \beta \circ f$. We will show that this implies $\alpha=\beta$.

By lemma \ref{lem-HIF}, we need to show that for every diagram
\[
\xymatrix{
\mathcal{I}\otimes\mathcal{I}_{x_j} \ar[r]^f & \mathcal{I} \ar@<-.5ex>[r]_{\alpha} \ar@<.5ex>[r]^{\beta} & \omega\otimes_k \mathcal{G}
}
\]
such that $\alpha\circ f = \beta\circ f$, we have $\alpha=\beta$. 
We will check this assertion on each piece of the open cover
\[
X\times_kY = ((X\setminus{x_j})\times_kY)\coprod (X^{\textrm{sm}}\times_kY)
\]
On the first piece the assertion is trivial since $\mathcal{I}_{x_j}$ is trivial. On the second piece the claim reduces to the smooth case. 
In that case, $\omega$ becomes a line bundle and it is easy to verify the claim stalk wise using the fact that $X$ is integral.
\end{proof}

We can now quickly deduce lemmas \ref{lem-induced-embedding-of-fibers} and \ref{lem-embedding-of-quot-is-cartier}.

\begin{proof}[Proof of lemma \ref{lem-induced-embedding-of-fibers}:]
We apply lemma \ref{lem-embedding-linear-systems} in the case where $T$ is the $k$ point $p_i$. The lemma now follows from lemma \ref{lem-epimorphism-of-sheaves}.
\end{proof}

\begin{proof}[Proof of lemma \ref{lem-embedding-of-quot-is-cartier}:]
We apply lemma \ref{lem-embedding-linear-systems} in the case where $T$ is $\Quot^{d+s-1}_{(\omega/X/S)}$ with the universal rank-1 torsion-free sheaf $\mathcal{I}_{d+s-1}$. By the choice of $s$, both $H(\mathcal{I}_{d+s-1}, \omega)$ and $H(\mathcal{I}_{d+s}, \omega)$ are locally free. Furthermore, by lemma \ref{lem-local-quot-structure}, the embedding is of codimension one. The lemma now follows from lemma \ref{lem-epimorphism-of-sheaves-good-kernel}.
\end{proof}

We will also need the following technical result about compositions of embeddings into a smooth scheme.
\begin{lem}\label{lem-composition-into-smooth}
Let $S$ be a $k$-scheme and let $X\xrightarrow{i} Y \rightarrow Z$ be closed embeddings of finite type $S$-schemes. Suppose $Z$ is a smooth $S$-scheme, $Y$ is a smooth $k$-scheme and $i$ is a regular embedding with normal bundle $N$. Then
\[
i^*s(Y,Z) = c(N)\cap s(X,Z)
\]
\end{lem}
\begin{proof}
Let $Z'=S\times_k Y$. The claim is trivial to verify when $Z=Z'$ and $Y\rightarrow Z'$ is the embedding of one of the fibers.
As we did earlier in the proof of proposition \ref{prop-smooth-case}, we will use the properties of Fulton's class from \cite[4.2.6]{Fulton} to compare general case to the case where $Z=Z'$. Specifically, we have
\[
\xymatrix @ R=0.5pc{
c(T_{Z/S})\cap s(X,Z) = c(T_{Z'/S})\cap s(X,Z') \\
c(T_{Z/S})\cap s(Y,Z) = c(T_{Z'/S})\cap s(Y,Z')
}
\]
from which we can deduce that
\[
\xymatrix @ R=0.5pc{
c(N) \cap s(X,Z) = s(T_{Z/S}) \cap c(T_{Z'/S})\cap c(N) \cap s(X,Z') \\
i^*s(Y,Z) = s(T_{Z'/S}) \cap c(T_{Z'/S})\cap i^*s(Y,Z')
}
\]

Since we know that the two terms on the right are equal, the lemma follows.
\end{proof}


\begin{thebibliography}{1}
  \bibitem{Fulton} William Fulton, \emph{Intersection Theory}
  \bibitem{Artin}  M. Artin, \emph{Algebraic Approximation of Structures Over Complete Local Rings}
  \bibitem{Kollar} Janos Kollar, \emph{Lectures on Resolution of Singularities}
  \bibitem{EJP} David Eklund, Christine Jost, Chris Peterson, \emph{A Method to Compute Segre Classes of Subschemes of Projective Space} 
  \bibitem{AK} Allen Altman, Steven Kleiman, \emph{Compactifying the Picard Scheme}
  \bibitem{CMK} Sebastian Casalania-Martin, Jesse Kass, \emph{A Riemann Singularity Theorem For Integral Curves}
  \bibitem{AIK} Allen Altman, A. Iarrobino, Steven Kleiman, \emph{Irreducibility of the Compactified Jacobian}
  \bibitem{S} R. L. E. Schwarzenberger, \emph{Jacobians and Symmetric Products}
  \bibitem{H} Robin Hartshorne, \emph{Deformation Theory}
  \bibitem{Edidin} Dan Edidin,  \emph{Strong regular embeddings and hypertoric geometry}, to appear
\end{thebibliography}
\end{document}